\documentclass[10pt]{article}
\usepackage{mathrsfs}
\usepackage{amsfonts}
\usepackage{amsthm,amsmath,amssymb,anysize}
\newtheorem{lemma}{Lemma}[section]
\newtheorem{theorem}[lemma]{Theorem}
\newtheorem{remark}[lemma]{Remark}
\newtheorem{proposition}[lemma]{Proposition}

\newtheorem{convention}[lemma]{Convention}
\usepackage[numbers,sort&compress]{natbib}
\marginsize{4.2cm}{4.2cm}{3.6cm}{3cm}%×óÓÒÉÏÏÂ
\numberwithin{equation}{section}

\title{\textsf{Superderivations for Modular Graded Lie  Superalgebras  of Cartan-type}}
\author{\textsc{Wei Bai}\;  \textsc{and} \textsc{Wende Liu}
\footnote{Correspondence: wendeliu@ustc.edu.cn (W. Liu)} \setcounter{footnote}{-1} \footnote{Supported by the NSF of China (10871057)
and the NSF of HLJP, China (A200802)} \\
  \\
{\small\textit{ School of Mathematical Sciences},
\textit{Harbin Normal University}} \\
{\small\textit{Harbin 150025, China}}
 }
\date{ }
\begin{document}
\maketitle
\begin{quotation}
\noindent\textbf{Abstract}  Superderivations for
  the eight families of finite or infinite dimensional graded Lie superalgebras of Cartan-type
over a field of characteristic $p>3$ are completely determined by a uniform approach: The infinite dimensional case is  reduced to the finite dimensional case and the latter is further reduced to the restrictedness case, which proves to be far more  manageable. In particular, the outer
superderivation algebras  of those Lie superalgebras are completely determined.\\
\noindent\textbf{Keywords}\ \ Lie  superalgebra; Cartan-type;
 superderivation

\noindent \textbf{Mathematics Subject Classification 2000}: 17B50, 17B40
  \end{quotation}

  \setcounter{section}{-1}

\section{Introduction}

Eight families of  $\mathbb{Z}$-graded Lie superalgebras of Cartan-type
were constructed  over a field of  characteristic $p>3$
  \cite{Fu-Zhang-Jiang,Liu-He,Liu-Zhang-Wang,Liu-Yuan-1,Zhang}. These Lie superalgebras are  subalgebras of the full
 superderivation algebras  of the associative superalgebras---tensor products  of the  divided power algebras and the
 exterior superalgebras. The superderivation
 algebras were  studied in one-by-one fashion for the \textit{finite dimensional and simple} ones \cite{Fu-Zhang-Jiang,Ma-Zhang,Liu-Zhang-Wang,Wang-Zhang,Zhang-Zhang}.
The present paper aims to use a uniform method to determine the surperderivation algebras of all the eight families of  graded Lie superalgebras of Cartan-type, including the \textit{infinite dimensional or non-simple} ones. In particular, the outer
superderivation algebras  of those Lie superalgebras are completely determined.
 We should mention that we adopt a method  for Lie algebras \cite[Lemma 6.1.3 ]{Strade} and benefit much from reading \cite{Strade,Strade-Farnsteiner}. It should be also mentioned that the present paper covers some known results about superderivations for the finite dimensional simple graded Lie superalgebras of Cartan-type mentioned above \cite{Fu-Zhang-Jiang,Ma-Zhang,Liu-Zhang-Wang,Wang-Zhang,Zhang-Zhang} and certain inaccuracies in the literature are corrected.

Throughout  $\mathbb{F}$ is an algebraically
closed field  of characteristic $p>3,$ $\mathbb{Z}_2:= \{\bar{0},\bar{1}\}$ is the field of two elements. As in usual, $\mathbb{Z},$ $\mathbb{N}$ and $\mathbb{N}_{0}$ are
the sets of integers, nonnegative integers and positive integers,
respectively. For a $\mathbb{Z}_2$-graded vector space $V$, denote by $|x|=\alpha$
the \textit{parity of a homogeneous element} $x\in V_{\alpha}$, $\alpha\in\mathbb{Z}_2$. If $V$
is a $\mathbb{Z}$-graded vector space and $x \in  V$ is a $\mathbb{Z}$-homogeneous element, write $\mathrm{zd}(x)$  for
the $\mathbb{Z}$-degree of $x$. The symbol  $|x|$ (resp. $\mathrm{zd}(x)$) always implies that $x$ is
a $\mathbb{Z}_2$-(resp. $\mathbb{Z}$-)homogeneous element.
\section{Basics}

 Fix two positive integers  $m$ and $n>1$. Let $\mathcal{O}(m)$ be the  \textit{divided power algebra} over $\mathbb{F}$ with
basis $\{x^{(\alpha)}\mid \alpha\in \mathbb{N}^{m}\} $ and $\Lambda (n)$  the
\textit{exterior superalgebra} over $\mathbb{F}$ with $n$ variables
$x_{m+1},\ldots ,x_{m+n}.$  The tensor product
$\mathcal{O}(m,n)
:=\mathcal{O}(m)\otimes_{\mathbb{F}} \Lambda(n)$
 is a super-commutative associative
superalgebra in the usual way.  For $g\in
\mathcal{O}(m),$ $f\in \Lambda(n),$   write $gf $ for
$ g\otimes f$.  Fix  two $m$-tuples of positive
integers $ \underline{t}:=\left( t_{1},t_{2},\ldots,t_m\right) $ and
$\pi:=\left( \pi _1,\pi _2,\ldots,\pi _m\right), $ where $\pi
_i:=p^{t_i}-1.$ The divided power algebra $\mathcal{O}(m)$ contains a finite dimensional subalgebra $\mathcal{O}(m;\underline{t}):=\mathrm{span}_{\mathbb{F}}\{ x^{( \alpha ) }\mid \alpha \in
\mathbb{A}(m;\underline{t})\},$  where
$\mathbb{A}(m;\underline{t}):=\left\{ \alpha \in
\mathbb{N}^m\mid\alpha _i\leq \pi _i \right\}$.
In particular,  $\mathcal{O}(m,n)$ has a finite dimensional subalgebra
$\mathcal{O}(m,n;\underline{t})
:=\mathcal{O}(m;\underline{t})\otimes_{\mathbb{F}} \Lambda(n)$.

Let
 $u:=\langle i_1,i_2,\ldots,i_k\rangle$ be a $k$-\textit{shuffle}, that is,  a strictly increasing sequence of  $k$ integers between
  $m+1$ and $m+n$.
Write $x^u:=x_{i_1}x_{i_2}\cdots x_{i_k}$ and $|u|:=k.$
  Notice that we also  denote the set $\{i_1,i_2,\ldots
,i_k\}$ by the $k$-shuffle  $u$ itself.
%Let $\mathbf{B}^0:=\left\{u\in\mathbf{B}(m)\mid |u|\quad\mbox{is even}\right\}$ and
%$\mathbf{B}^1:=\left\{u\in\mathbf{B}(m)\mid |u|\quad\mbox{is odd}\right\}$.
%For $u,v\in \mathbf{B}(n) $ with $u\cap v=\emptyset,$  define $u+v$ to be $w\in \mathbf{B}(n)$ such that
%and $ w = u \cup v.$ If $v\subset u,$ define $u-v$ to be $w\in \mathbf{B}(n)$ such that $w=u\setminus v.$
%For $u\in \mathbf{B}(m),$
The only $n$-shuffle is $\omega:=\langle m+1,\ldots, m+n\rangle$. For short, put
$\mathbf{I}_0:=\overline{1,m},$ $\mathbf{I}_{1}:=\overline{m+1,m+n}$
and $\mathbf{I}:=\overline{1,m+n}.$ For a proposition $P$, put
$\delta_{P}:=1 $ if $ P$  is true and $\delta_{P}:=0$ otherwise.
 For $\varepsilon
_i:=( \delta_{i1}, \ldots,\delta _{im}),$ we abbreviate
$x^{(\varepsilon _i)}$ to $x_i $ for $i\in \mathbf{I}_0.$ Let
$\partial_i$ be the Special superderivation of
$\mathcal{O}(m,n) $  such that
$\partial_i(x_j)=\delta_{ij}$ for $i, j\in \mathbf{I}$.

From now on, we adopt the convention $(m,n;\underline{\infty})=(m,n).$ For example, we have
$\mathcal{O}(m,n;\underline{\infty})=\mathcal{O}(m,n).$
Let us introduce  the eight families of $\mathbb{Z}$-graded Lie superalgebras of Cartan-type as follows \cite{Fu-Zhang-Jiang,Liu-He,Liu-Zhang-Wang,Liu-Yuan-1,Zhang}.
 In the below $t$ may be  $\infty$.
\begin{itemize}
\item[(1.1)]
The \textit{generalized Witt superalgebra}  $
W\left(m,n;\underline{t}\right)$ is spanned by all $f_r \partial_r,$
where $ f_r\in \mathcal{O}(m,n;\underline{t}),$ $r\in\mathbf{I}$.

\item[(1.2)]
Let $\mathrm{div}: W(m,n;\underline{t})\longrightarrow \mathcal{O}(m,n;\underline{t})$
be the  {divergence}, which is an even linear operator  such that
$\mathrm{div}(f\partial_k)=(-1)^{|\partial_k||f|}\partial_k(f)$ for all $ k\in\mathbf{I}.$
 %Note that $\mathrm{div}$ is an even $\mathbb{Z}$-homogeneous superderivation of $W(m,n;\underline{t})$
 %into the module $\mathcal{O}(m,n;\underline{t})$.
%\begin{eqnarray*}%\label{ct201011151947}
%\mathrm{div}[E,D]=E(\mathrm{div} D)-(-1)^{|E||D|}D(\mathrm{div } E)\quad\mbox{for all}\ D, E\in W(m,m;\underline{t}).
%\end{eqnarray*}
For $i, j\in \mathbf{I}$, let $D_{ij}:
\mathcal{O}(m,n;\underline{t})\longrightarrow W(m,n;\underline{t})$
be a linear operator  such that for $a\in\mathcal{O}(m,n;\underline{t}),$
 $D_{ij}(a):=(-1)^{|\partial_i||\partial_j|}\partial_i(a)\partial_j-(-1)^{(|\partial_i|+|\partial_j|)|a|}\partial_j(a)
\partial_i.
$
 The \textit{Special
superalgebra} is
\begin{itemize}
\item[]$S(m,n;\underline{t}):=\{D \in
W(m,n;\underline{t})\mid\mathrm{div}(D)=0\}.$
\end{itemize}
It is the derived algebra of
\begin{itemize}
\item[]$\overline{S}(m,n;\underline{t}):=\{D \in
W(m,n;\underline{t})\mid\mathrm{div}(D)\in \mathbb{F}\}. $
\end{itemize}
Moreover, the derived algebra of $S(m,n;\underline{t})$,
\begin{itemize}
\item[]$S(m,n;\underline{t})^{(1)}=\mathrm{span}_{\mathbb{F}}\{D_{ij}(a)\mid
a\in\mathcal{O}(m,n;\underline{t}),\, \, i, j\in \mathbf{I} \},$
\end{itemize}
 is a simple Lie superalgebra.
\end{itemize}
Write $m=2r$ or $2r+1$. Let $'$ be the involution of $\mathbf{I}$ such $i'=i+r$ for $i\in\overline{1, r}$ and $i'=i$ for $i\in\mathbf{I}_{1}.$ We also use the mapping $\sigma:\mathbf{I}\longrightarrow \{1, -1\}$ given by
$\sigma(i)=-1$ for $i\in \overline{r+1, 2r}$ and $\sigma(i)=1$ otherwise.
\begin{itemize}

\item[(1.3)]
Suppose $m=2r$ is even.
Let $D_{H}: \mathcal{O}(m,n;\underline{t})\longrightarrow
W(m,n;\underline{t})$ be an even linear operator  given by
$D_{H}(a):=\sum_{i\in
\mathbf{I}}\sigma(i)(-1)^{|\partial_i||a|}\partial_i(a)\partial_{i'}.$
The \textit{Hamiltonian superalgebra} is
\begin{itemize}
\item[]$H(m,n;\underline{t})=\mathrm{span}_{\mathbb{F}}\{D_{H}(a)\mid
a\in\mathcal{O}(m,n;\underline{t})\}.$
\end{itemize}
Its derived algebra is  simple.
While $H(m,n;\underline{t})$ is the derived algebra of the Lie superalgebra
\begin{itemize}
\item[]$\overline{H}(m,n;\underline{t}):=\overline{H}(m,n;\underline{t})_{\bar0}\oplus
\overline{H}(m,n;\underline{t})_{\bar1},$
\end{itemize}
 where for $\alpha \in \mathbb{Z }_2,$
\begin{eqnarray*}\begin{split}
\overline{H}(m,n;\underline{t})_{\alpha}=&\bigg
\{\sum_{i\in \mathbf{I}} a_i\partial_i\in
  W(m,n;\underline{t})_{\alpha}\bigg|\\
 &\partial_i(a_{j'})=(-1)^{|\partial_i||\partial_j|+(|\partial_i|+|\partial_j|)\alpha}\sigma(i)\sigma(j)\partial_j(a_{i'}), i, j \in
 \mathbf{I}\bigg\}.\end{split}
\end{eqnarray*}
%Note that $D_{H}$ is even,
%with kernel $\mathbb{F}\cdot 1$, $\mathbb{Z}$-degree $-2$.
Write $\bar {\mathcal{O}}(m,n;\underline{t}) $ for the quotient
superspace $\mathcal{O}(m,n;\underline{t})/\mathbb{F}\cdot 1$ and
view $D_H$ as the linear operator of $\bar {\mathcal{O}}(m,n;\underline{t})$. One sees
that $H(m,n;\underline{t})\cong(\bar
{\mathcal{O}}(m,n;\underline{t}), [\;,\;]_{H})$, where the
bracket is:
 $ [a,b]_{H}:= D_H\left(a\right)\left(b\right)$ for
   $a,b\in\bar{\mathcal{O}}(m,n;\underline{t}).$
%Notice that $H(m,n;\underline{t})^{(2)}$ is a simple Lie superalgebra.

\item[(1.4)]
Suppose $m=2r+1$ is odd. The \textit{contact superalgebra} is by definition
\begin{itemize}
\item[]$K(m,n;\underline{t}):=\mathrm{span}_{\mathbb{F}}\{D_K(a)\mid a\in\mathcal{O}(m,n;\underline{t})\}.$
\end{itemize}
$$D_{K}(a):=\sum_{i\in \mathbf{I}\backslash\{m\}}(-1)^{|\partial_i||a|}\Big(x_i\partial_m(a)+\sigma(i')\partial_{i'}(a)\Big)\partial_i
+\Big(2a-\sum_{i\in \mathbf{I}\backslash\{m\}}x_i\partial_i(a)\Big)\partial_m.$$
We have a Lie superalgebra isomorphism
\begin{itemize}
\item[]$K(m,n;\underline{t})\cong(\mathcal{O}(m,n;\underline{t}),[\;,\;]_K),$
\end{itemize}
where the Lie bracket is:  $ [a,b]_K=
D_K(a)\left(b\right)-2\partial_m(f)(g)$. Note that
 $K(m,n;\underline{t})^{(1)}$ is  simple.

%As in \cite{w}, for any $m$-tuple $\gamma\in\mathbb{N}^m$, a direct computation shows that
%${\mathcal{O}}(m,n;\underline{t})$ and $W(m,n;\underline{t})$ are graded by means of
%\begin{eqnarray*}
%\mathcal{O}(m,n;\underline{t})_{\gamma, i}&=&\mathrm{span} _{\mathbb F}\bigg\{ x^{(\alpha)} x^{u}\mid |\alpha\gamma|+|u|=i \bigg\},\\
%W(m,n;\underline{t})_{\gamma, i}&=&\mathrm{span} _{\mathbb F}\bigg\{f\partial_k\mid f\in \mathcal{O}(m,n;\underline{t})_{\gamma, i+\gamma_k} \bigg\},
%\end{eqnarray*}
%where $\gamma_k=1$ for $k\in \mathbf{I}_1$.
% Putting $X(m,n;\underline{t})_{\gamma, i}:=X(m,n;\underline{t})\cap W(m,n;\underline{t})_{\gamma, i}$, where
% $X=W, S, H, K$. In the sequel, we will always assume that $\gamma=\underline{1}+\delta_{X, K}\varepsilon_m$ and omit the subscript $\gamma$
%since $X(m,n;\underline{t})_{\gamma, i}$ is a graded subalgebra of $W(m,n;\underline{t})_{\underline{1}+\delta_{X, K}\varepsilon_m}$.
\end{itemize}

In the below, we introduce  the other four families of Lie superalgebras of Cartan-type. In these cases, suppose $m>2$
and $n=m$ or $m+1$. Let $\tilde{}$ be the involution of $\textbf{I}$ such that $\tilde{i}= i+m$ for $
i\in\mathbf{I}_{0}$.
When $n=m$, from \cite{Liu-He,Liu-Zhang-Wang} we have the following two families  of Lie superalgebras.
\begin{itemize}
\item[(1.5)]
Define an odd linear operator
$T_H:\mathcal{O}(m,m;\underline{t}) \longrightarrow  W(m,m;\underline{t}) $ such that
 $T_H(a):=\sum_{i\in  \mathbf{I}}(-1)^{|\partial_i||a|}\partial_i(a)\partial_{\tilde{i}}$
for $a\in \mathcal{O}(m,m;\underline{t}).$
% Note that $T_H$ is odd, with kernel $\mathbb{F}\cdot 1$ and $\mathbb{Z}$-degree-2.
The \textit{odd Hamiltonian superalgebra} is
\begin{itemize}
\item[]$ HO(m;\underline{t}):=\mathrm{span}_{\mathbb{F}}\{T_{H}(a)\mid a\in\mathcal{O}(m,m;\underline{t})\},$
\end{itemize}
 which is  simple.
 It is the derived algebra of the Lie superalgebra
 \begin{itemize}
\item[]$\overline{HO}(m;\underline{t}):=\overline{HO}(m;\underline{t})_{\bar0}\oplus
\overline{HO}(m;\underline{t})_{\bar1},$
\end{itemize}
 where for $\alpha \in \mathbb{Z }_2,$
\begin{eqnarray*}\begin{split}
\overline{HO}(m;\underline{t})_{\alpha}=&\bigg \{\sum_{i\in \mathbf{I}} a_i\partial_i\in
  W(m,m;\underline{t})_{\alpha}\bigg|\\
 &\partial_i(a_{\tilde{j}})=(-1)^{|\partial_i||\partial_j|+(|\partial_i|+|\partial_j|)(\alpha+\bar{1})}\partial_j(a_{\tilde{i}}), i, j \in
 \mathbf{I}\bigg\}.\end{split}
\end{eqnarray*}
%From \cite[Proposition 20]{Liu-Zhang-Wang}, we know that
%$$\overline{\overline{HO}}(m,m;\underline{t})=HO(m;\underline{t})\oplus\mathrm{span}_{\mathbb{F}}
%\big\{x^{(\pi_i\varepsilon_i)}\partial_{i'}\mid i\in \mathbf{I}_0\big\} $$
We have a Lie superalgebra isomorphism $HO(m;\underline{t})\cong(\bar
{\mathcal{O}}(m,m;\underline{t}), [\;,\;]_{HO})$, where the Lie
bracket is,  $[a,b]_{HO}:={T_H}\left(a\right)\left(b\right)$ for
$a,b\in\bar{\mathcal{O}}(m,m;\underline{t}).$
\item[(1.6)] The \textit{Special odd Hamiltonian superalgebra} is
 $SHO(m;\underline{t}):=S(m,m;\underline{t})\cap HO(m;\underline{t})$.
Its second  derived superalgebra is simple.  Put
$\overline{SHO}(m;\underline{t}):=\overline{S}(m,m;\underline{t})\cap\overline{HO}(m;\underline{t}).$
%It is easy to see that $HO(m;\underline{t}^{(1)})$ has a $\mathbb{Z}$-grading structure induced by $\mathrm{zd}(x_i)=-1$.
\end{itemize}
When $n=m+1$,  from  \cite{Fu-Zhang-Jiang,Liu-Yuan-1}  we have the following
two families of Lie superalgebras.
\begin{itemize}
\item[(1.7)]
The \textit{odd Contact superalgebra} is
\begin{itemize}
\item[]$KO(m;\underline{t}):=\mathrm{span}_{\mathbb{F}}\{D_{KO}(a)\mid a\in\mathcal{O}(m,m+1;\underline{t})\},$
\end{itemize}
where $D_{KO}:\mathcal{O}(m,m+1;\underline{t}) \longrightarrow  W(m,m+1;\underline{t})$ is  given by
\begin{itemize}
\item[]$D_{KO}(a):
=T_{H}(a)+(-1)^{|a|}\partial_{2m+1}(a)\mathfrak{D}+
\big(\mathfrak{D}(a)-2a \big)\partial_{2m+1}.$
\end{itemize}
Hereafter, $\mathfrak{D}:=\sum
_{i=1}^{2m}x_{i}\partial_{i}. $
% Here $\mathfrak{D}$ is the degree superderivation of $\mathcal{O}(m,m;\underline{t})$.
 Note that $KO(m;\underline{t})$ is  simple and
  $KO(m;\underline{t})\cong(\mathcal{O}(m,m+1;\underline{t}),[\;,\;]_{KO})$,
where the bracket is
\begin{equation*}
[a,b]_{KO}=
D_{KO}\left(a\right)\left(b\right)-(-1)^{|a|}2\partial_{2m+1}\left(a\right)b\quad
\mbox{for}\; a,b\in \mathcal{O}(m,m+1;\underline{t}).
\end{equation*}
\item[(1.8)]
 Given $\lambda\in \mathbb{F},$ for $a\in \mathcal{O}(m,m+1;\underline{t})$ $m>3$, consider the linear operator $\mathrm{div}_{\lambda}$:
$$
\mathrm{div}_{\lambda}(a):=(-1)^{|a|}2\left(\sum_{i=1}^{m}\partial_i\partial_{\tilde{i}}
\left(a\right)+\left(\mathfrak{D}-m\lambda\mathrm{id}_{\mathcal{O}(m,m+1;\underline{t})}\right)
\partial_{2m+1}\left(a\right)\right).$$

The  kernel of $\mathrm{div}_{\lambda}$ is called the \textit{Special odd Contact superalgebra}, denoted by
$SKO(m;\underline{t})$. Its second derived algebra is simple.
\end{itemize}

%It is easy to see that $KO(m,m;\underline{t}$ has a $\mathbb{Z}$-grading structure induced by  $\mathrm{zd}(1)=-2$ and $\mathrm{zd}(x_i)=-1+\delta_{i, 2m+1}$.

\begin{convention}\label{ctc101120} Hereafter $X$ denotes  $W$, $S$, $H$, $K$,
$HO$, $SHO$, $KO$ or $SKO$. For simplicity we usually write  $X(\underline{t})$  for $X(m,n;\underline{t})$ $(X=W, S, H$ or $K)$ and $X(m;\underline{t})$ $(X=HO, SHO, KO$ or $SKO)$, where $t$ is $\infty$ or not.
It is convenience to  identify   $X(\underline{t})$ with a  finite dimensional
  subalgebra of  $X(\underline{\infty})$ for $t\not= {\infty}$.
  \end{convention}

 $X(\underline{t})$ and its derived algebras are referred to as the \textit{graded Lie superalgebras of Cartan-type}.
%Suppose $L=\sum_{i=l}^hL_i$ is a $\mathbb{Z}$-graded space.
%We call $h(l)$  \textit{the highest(lowest) degree} in the gradation of $L$.

  \begin{remark}\label{ctr110310}  When $\underline{t}=\underline{\infty}$,
  $S(\underline{t})$,   $H(\underline{t})$,  $K(\underline{t})$,  $SHO(\underline{t})^{(1)}$ and $SKO(\underline{t})^{(1)}$
  are simple.
\end{remark}

\section{Reduction}

In this section we
establish some technical lemmas to simplify our consideration. Propositions \ref{ct101120l2.1} and \ref{ct11012110582.2}
  play an  important role for  determining the superderivations of   Lie superalgebras of Cartan-type.
  For later use we first list the heights of the graded Lie superalgebras of Cartan type.

  \begin{remark}\label{ctr101120}   $X(\underline{t})$ has a \textit{principal  grading}
 satisfying that
$$\mathrm{zd}(x_i)=-\mathrm{zd}(\partial_i)=  1+\delta_{X=K}\delta_{i=m}+\delta_{X=KO}\delta_{i=2m+1}+\delta_{X=SKO}\delta_{i=2m+1}.$$
Let $h(X)$ denote the height of
$X(\underline{t})^{(2)}$ and put
$\xi(\underline{t}):=\sum_{i=1}^mp^{t_i}-m+n.$  \\

   ~~~~~~~~~~\begin{tabular}{|l|l|}
   \multicolumn{2}{c}{$\rm{Heights\; of\; Lie\; superalgebras\; of\; Cartan\; type}$}

    \\[5pt] \hline
     \multicolumn{1}{|c|}{$\mathrm{Height}$ $h(X)$}&
     \multicolumn{1}{|c|}{$\mathrm{Lie\; superalgebra}$  $X$}\\
     \hline
     $\xi(\underline{t})-1$ & $W, \, KO$  \\
     \hline
     $\xi(\underline{t})-2$ & $S, \, HO$, $SKO $  $\mathrm{ with}$   $m\lambda+1\neq 0$  in $\mathbb{F}$ \\
     \hline
     $\xi(\underline{t})-3$ & $H, \,  SKO \; \mathrm{with} \; m\lambda+1= 0\;  \mbox{in}\; \mathbb{F}$ \\
     \hline
     $\xi(\underline{t})-5$ & $SHO$ \\
     \hline
     $\xi(\underline{t})+p^{t_m}-3$ & $K  \;\mathrm{with}\; n-m-3\neq 0 \;\mbox{ in } \mathbb{F}$ \\
     \hline
     $\xi(\underline{t})+p^{t_m}-4$ & $K  \; \mathrm{with} \; n-m-3= 0 \;\mbox{ in } \mathbb{F}$ \\
     \hline
   \end{tabular}
\end{remark}

Suppose   $L$ is a  Lie superalgebra and   $V$ is an $L$-module.
Denote by $\mathrm{Der}(L, V)$ the \textit{superderivation space}
 and
$\mathrm{Inder}(L, V)$ the \textit{inner derivation space}.   Clearly, $\mathrm{Der}(L, V)$ is an
$L$-submodule of $\mathrm{Hom}_{\mathbb{F}}(L, V)$.
Assume in addition that   $L=\bigoplus_{r\in\mathbb{Z}}L_r$ is
$\mathbb{Z}$-graded and finite-dimensional, and $V=\bigoplus_{r\in\mathbb{Z}}V_r$ is a
$\mathbb{Z}$-graded $L$-module. Then the superderivation space inherits a
$\mathbb{Z}$-graded $L$-module structure
\begin{itemize}
\item[]$\mathrm{Der}(L, V)=\bigoplus_{r\in\mathbb{Z}}\mathrm{Der}_r(L, V).$
\end{itemize}
As in the usual, write
\begin{itemize}
\item[]$\mathrm{Der}^-(L, V):=\mathrm{span}_{\mathbb{F}}\{\phi\in \mathrm{Der}_i(L, V)\mid i<0\}.$
\end{itemize}
Let $T\subset L_0\cap L_{\bar{0}}$ be a torus of $L$  with the weight
space decompositions:
\begin{itemize}
\item[]$L=\oplus_{\alpha\in \Theta}L_{(\alpha)},\quad  V=\oplus_{\beta\in \Delta}V_{(\beta)}.$
\end{itemize}
Then there exist subsets
$\Theta_{i}\subset\Theta$ and $\Delta_{j}\subset\Delta$ such that
$L_{i}=\oplus_{\alpha\in\Theta_{i}}L_{i}\cap L_{(\alpha)}$ and $V_{j}=\oplus_{\beta\in\Delta_{j}}V_{j}\cap V_{(\beta)}.$
Hence $L $ and $V$ have the corresponding $\mathbb{Z}\times T^{*}$-grading structures, respectively.
Of course $\mathrm{Der}(L, V)$ inherits a
$\mathbb{Z}\times T^{*}$-grading from $L$ and
$V$ as above.
A superderivation $\phi\in\mathrm{Der}(L, V)$ is  call a \textit{weight-derivation} if it is $T^*$-homogeneous. Write $\theta$ for the zero weight.
\begin{lemma}\label{obl201010171415}
A weight-derivation $\phi\in\mathrm{Der}(L, V)$ is inner if it is a
nonzero weight-derivation. In particular, any derivation
$\psi\in\mathrm{Der}(L, V)$  is inner  modulo a derivation of zero
weight-derivation.
\end{lemma}
\begin{proof}
Suppose  $\phi\in \mathrm{Der}_{(\alpha)}(L, V)$ and $\alpha\not=\theta$.
Since $\mathrm{Der}_{(\alpha)}(L, V)$ is $\mathbb{Z}$-graded, write $\phi=\sum_{i\in\mathbb{Z}}\phi_i$, where  $\phi_i\in \mathrm{Der}_{(\alpha)}(L, V)$.
Then there exists $t \in T$ with $\alpha(t)\neq 0$ such that for
arbitrary $x \in L$,
\begin{itemize}
\item[]$\alpha(t)\phi_{i}(x)=(t\cdot\phi_{i})(x)=t\cdot\big(\phi_{i}(x)\big)-\phi_{i}\big([t, x]\big)=x\cdot\big(\phi_{i}(t)\big).$
\end{itemize}
Hence $\phi_{i}(x)=x\cdot\big(\alpha(t)^{-1}\phi_i(t)\big)$, which implies that $\phi_{i}$ is inner.
Then  $\phi$ is inner.
\end{proof}

Analogous to \cite[Proposition 3.3.5 and Lemma 4.7.1]{Strade-Farnsteiner}, we have
\begin{lemma}
\label{ct101120l2.2} Let $L=\oplus_{i=-r}^{h}L_i$ be a
finite dimensional $\mathbb{Z}$-graded simple Lie superalgebra. The following statements hold:
\begin{itemize}
\item[$\mathrm{(1)}$]$L_{-r}$ and $L_h$ are irreducible $L_0$-modules.

\item[$\mathrm{(2)}$] $[L_0, L_h]=L_h,$ \quad $[L_0, L_{-r}]=L_{-r}.$

\item[$\mathrm{(3)}$] $C_{L_{h-1}}(L_1)=0,$ \quad $[L_{h-1}, L_1]=L_h.$

\item[$\mathrm{(4)}$] $C_L{(\oplus_{i>0}L_i)}=L_h,$ \quad $C_L{(\oplus_{i<0}L_i)}=L_{-r}.$
\end{itemize}
 If $M\subset L$ is a subalgebra containing $L_{-1}\oplus L_1$ and if $M\cap L_{h-1}\not=0$,
 then $M=L$.\qed
\end{lemma}

Analogous to
\cite[Lemma 2.1.3]{Liu-Zhang-1}, we have

\begin{lemma}\label{ct201101201245}
Let V be an arbitrary vector superspace over $\mathbb{F}$.  Suppose $A_1$, $A_2$,\ldots, $A_k\in \mathrm{End}_{\mathbb{F}}V$
  span an abelian sub-Lie superalgebra of $\frak{gl}(V)$. Suppose further each $A_i$ is generalized invertible,
that is, there is $B_i\in \mathrm{End}_{\mathbb{F}}V$ with $|B_i|=|A_i|$ such that

\begin{itemize}
\item[]$A_iB_iA_i=A_i, \, 1\leq i\leq k,$
\item[] $A_iB_j=(-1)^{|A_i||B_i|}B_jA_i, \, i\leq i\not=j\leq k.$
\end{itemize}
If $v_1, v_2, \ldots, v_k\in V$ satisfy:
\begin{itemize}
\item[]$A_iB_i(v_i)=v_i, \, 1\leq i\leq k,$
\item[] $A_i(v_j)=(-1)^{|A_i||A_j|}A_j(v_i), \, 1\leq i, j\leq k,$
\end{itemize}
then there exists $v\in V$ such that $A_i(v)=v_i$ for all $1\leq i\leq k.$\qed
\end{lemma}
For $i\in \mathbf{I}$, define a linear operator  $\Phi_i: {W}(\underline{t})\longrightarrow {W}(\underline{t})$, such that
\begin{itemize}
\item[]$\Phi_i(x^{(\alpha)}x^u\partial_j):= \left\{\begin{array}{lll} x^{(\alpha+\varepsilon_i)}x^u\partial_j, &\mbox{if}\;
i\in\mathbf{I}_0;
\\x^{(\alpha)}x_ix^u\partial_j, &\mbox{if}\; i\in \mathbf{I}_1.
\end{array}\right.$
\end{itemize}
When $\underline{t}\not=\underline{\infty}$ we adopt the convention that $x^{(\alpha+\varepsilon_i)}=0$ whenever $\alpha+\varepsilon_i\not\in\mathbb{A}(m;\underline{t}).$
Clearly, $\Phi_i $ is of $\mathbb{Z}$-degree 1 and $|\Phi_i|=|\partial_i|$. An element $D\in W(\underline{t})$
is called $i$-\textit{integral} provided that $\partial_i\Phi_i(D)=D.$

\begin{lemma}\label{ct201101201402} Let
$L$ be a $\mathbb{Z}$-graded subalgebra of ${W}(\underline{t})$ with depth $r$ such
that $ L_{{-r}}=\mathrm{span}_{\mathbb{F}}\{\partial_{j}\mid j\in
\mathbf{J}(k)\}$ for some $k\in \mathbf{I}$, where
$\mathbf{J}(k)=\{i_1, \ldots, i_k\}$ is the set of $k$ integers in $\mathbf{I}$.
  Then  $\phi(\partial_{j})$ is ${j}$-integral for any  $\phi\in \mathrm{Der}{\big(L, {W}(\underline{t})\big)}$. Furthermore, there exists $D\in {W}(\underline{t})$ such that
$\phi-\mathrm{ad}D$ vanishes on $L_{-r}$.
\end{lemma}
\begin{proof}
For ${j}\in \mathbf{I}_1$, we have $
 [\partial_{j}, \phi(\partial_{j})]=0$ and $\phi(\partial_{j})$ is  ${j}$-integral.
Suppose ${j}\in \mathbf{I}_0$. It is clear that
  the elements of ${W}(\underline{\infty})$ are ${j}$-integral for all ${j}\in \mathbf{I}_0$.
For $\underline{t}\not=\underline{\infty}$,  $\mathrm{Der}(W)$ is a
restricted Lie superalgebra with respect to the $p$-power and
consequently,
\begin{itemize}
\item[]$\mathrm{ad}\big((\mathrm{ad}\partial_{j})^{p^{t_{j}}-1}\phi(\partial_{j})\big)=[\phi, (\mathrm{ad}\partial_{j})^{p^{t_{j}}}]=0.
$
\end{itemize}
Since ${W}(\underline{t})$ is simple, we have
$(\mathrm{ad}\partial_{j})^{p^{t_{j}}-1}\phi(\partial_{j})=0$
and then
  $\phi(\partial_{j})$ is ${j}$-integral.
Without loss of generality one may assume that $\phi$ is
$\mathbb{Z}_2$-homogeneous. For $\partial_{j},
\partial_{s}\in L_{-r}$, since $[\partial_{j},
\partial_{s}]=0$, it follows that
\begin{itemize}
\item[]$\mathrm{ad}\partial_{j}\big((-1)^{|\phi||\partial_{s}|}\phi(\partial_{s})\big)=
(-1)^{|\partial_{j}||\partial_{s}|}
\mathrm{ad}\partial_{s}\big((-1)^{|\phi||\partial_{j}|}\phi(\partial_{j})\big)$
\end{itemize}
Put $V=W(\underline{s})$, $A_j=\mathrm{ad}\partial_{j}$,
$B_j=\Phi_j$ and
$v_j=(-1)^{|\phi||\partial_{j}|}\phi(\partial_{j})$.
 Form Lemma \ref{ct201101201245},  we can find $D\in {W}(\underline{s})$ such that
$\phi-\mathrm{ad}D$ vanishes on $L_{-r}$.
\end{proof}

For short, write $L^-$ for $\sum_{i<0}L_{i}$ when $L$ is a $\mathbb{Z}$-graded Lie superalgebra.

\begin{proposition}\label{ct101120l2.1}
For any $\phi \in \mathrm{Der}\big(X(\underline{t}),
{W}(\underline{t})\big)$, there exists $D\in {W}(\underline{t})$ such
that $\phi-\mathrm{ad}D$ vanishes on $X^{-}(\underline{t})$. If $X(\underline{t})$ is finite dimensional and $\phi \in
\mathrm{Der}_k\big(X(\underline{t}), {W}(\underline{t})\big)$ for $k\geq
-r+1$, where $r$ is the depth of
$X(\underline{t})$, then there exists $D\in {W}(\underline{t})_k$
such that $\phi-\mathrm{ad}D$ vanishes on $X(\underline{t})$.

\end{proposition}
\begin{proof}
Without loss of generality one may assume that $\phi$ is
$\mathbb{Z}_2$-homogeneous. Obviously, the conclusions hold for
$X=W, S, H, HO$ or $ SHO$.

For $X=K$, by Lemma \ref{ct201101201402}, there exists $D_1\in
{W}(\underline{t})$ such that $(\phi-\mathrm{ad}D_1)(1)=0$. Put
$\varphi=\phi-\mathrm{ad}D_1$. Then  for any $r, q\in
\mathbf{I}\backslash\{m\}$, we have
$\partial_m\big(\varphi(x_r)\big)=0.$ It follows that
\begin{eqnarray}
\sigma(r')\partial_r\big((-1)^{(|\phi|+1)|\partial_q|}\varphi(x_{q'})\big)
=(-1)^{|\partial_r||\partial_q|}
\sigma(q')\partial_q\big((-1)^{(|\phi|+1)|\partial_r|}\varphi(x_{r'})\big).\label{ct1102171422}
\end{eqnarray}
Put $A_r=\sigma(r')\mathrm{ad}\partial_r$, $B_r=\Phi_r$ and
$v_r=(-1)^{(|\phi|+1)|\partial_r|}\varphi(x_{r'})$ for
$r\in\mathbf{I}\backslash{m}$.
 From (\ref{ct1102171422}) we have  by  a direct computation that $v_r$ is  $r$-integral.
By Lemmas \ref{ct201101201245} and \ref{ct201101201402},  there exists $D_2\in {W}(\underline{t})$
 with $\partial_m(D_2)$=0 such that
$A_r(D_2)=v_r$. It is easy to prove that there exists $D\in
{W}(\underline{t})$ such that $\phi-\mathrm{ad}D$ vanishes on
$K^{-}(\underline{t})$.

For $X=KO$ or $SKO$, the proof is similar. When
$X(\underline{t})$ is finite dimensional, by induction on $i$ and
Lemma \ref{ct101120l2.2}(4) we obtain that $\phi-\mathrm{ad}D$
vanishes on $X(\underline{t})$ and $D\in W(\underline{t})_{k}.$
 \end{proof}

\begin{proposition}\label{ct11012110582.2}
Let $\phi\in\mathrm{Der}\big(X(\underline{\infty})\big) $.
If $\phi\big(X^{-}(\underline{\infty})\big)=0$, then $\phi$ leaves $X(\underline{t})$ invariant for any
 $\underline{t}\not=\underline{\infty}$.

\end{proposition}
\begin{proof}
Let $E\in X(\underline{\infty})$. A sufficient and necessary condition for $E\in X(\underline{t})$ is  that
\begin{itemize}
\item[$\mathrm{(1)}$] for $X=W, S, H, HO$ or $ SHO$,
 $(\mathrm{ad}\partial_i)^{p^{t_i}}(E)=0$  for all $i\in \mathbf{I}_0;$
\item[$\mathrm{(2)}$] for $X=KO$ or $SKO$,
 $(\mathrm{ad}x_i)^{p^{t_i}}(E)=0$ for all $i\in \mathbf{I}_1\backslash\{2m+1\};$
\item[$\mathrm{(3)}$] for $X=K$,
 $(\mathrm{ad}x_i)^{p^{t_i}}(E)=(\mathrm{ad}1)^{p^{t_m}}(E)=0$ for all $i\in \mathbf{I}_0\backslash\{m\}.$
\end{itemize}
Since $\phi(X^{-}(\underline{\infty}))=0$
  we  obtain that
\begin{itemize}
\item[] $[\phi, \mathrm{ad}\partial_i]=0$  for all $i\in \mathbf{I}_0
\mbox{ and } X=W, S, H, HO \mbox{ or } SHO;$
\item []
 $[\phi, \mathrm{ad}x_i]=0$  for all $i\in \mathbf{I}_1\backslash\{2m+1\}$
   and $X=KO$ or $SKO;$
  \item[]
 $[\phi, \mathrm{ad}x_i]=[\phi, \mathrm{ad}1]=0$  for all
 $i\in \mathbf{I}_0\backslash\{m\} $ and $X=K.$
\end{itemize}
Thus we have $\phi(X(\underline{t}))\subset X(\underline{t})$.
\end{proof}

In order to prove the next proposition, we  establish a technical lemma.
\begin{lemma}\label{ct1101211513} Let $L=HO, SHO, KO$ or $SKO$. Suppose $A_L$ is a subalgebra of $L(\underline{\infty})$ and $\underline{s}\geq \underline{1}$.
If $L(\underline{s})^{(2)}+\mathbb{F}x^{((p^{s_i+1}-1)\varepsilon_i)}\subset A_L$, then $L(\underline{s}+\varepsilon_i)^{(2)}\subset A_L$.
\end{lemma}
\begin{proof}
Note that
\begin{eqnarray*}
\begin{split}
 &f_{HO}:=x^{(\pi-(p^{s_i}-2)\varepsilon_i)}x^\omega \in HO(\underline{s}); \quad
   f_{KO}:=x^{(\pi-(p^{s_i}-2)\varepsilon_i)}x^\omega \in KO(\underline{s});\\
 &f_{SHO}:=\sum_{j=1}^mx^{(\pi-(p^{s_i}-2)\varepsilon_i-\varepsilon_j)}\partial_{\tilde{j}}(x^{\omega}) \in SHO(\underline{s});\\
 &f_{SKO}:=\sum_{j=1}^mx^{(\pi-(p^{s_i}-2)\varepsilon_i-\varepsilon_j)}\partial_{\tilde{j}}(x^\omega)\\
&\quad\quad\quad\quad+(-1)^{m-1}m\lambda
x^{(\pi-(p^{s_i}-2)\varepsilon_i)}\partial_{\tilde{j}}(x^{\omega-\langle{2m+1}\rangle})\in SKO(\underline{s}).
\end{split}
\end{eqnarray*}
Computing $\big[x^{((p^{s_i+1}-1)\varepsilon_i)},
f_{X}\big]_{X}$, one gets
\begin{itemize}
\item[$\mathrm{(1)}$] $x^{(\pi+(p-1)p^{s_i}\varepsilon_i)}x^{\omega-\langle{\tilde{i}}\rangle}\in A_{HO}\cap HO(\underline{s}+\varepsilon_i)^{(2)}_{h(HO)-1}$;
\item[$\mathrm{(2)}$] $x^{(\pi+(p-1)p^{s_i}\varepsilon_i)}x^{\omega-\langle{\tilde{i}}\rangle}\in A_{KO}\cap KO(\underline{s}+\varepsilon_i)^{(2)}_{h(KO)-1}$;
\item[$\mathrm{(3)}$] $\sum_{j=1}^mx^{(\pi+(p-1)p^{s_i}\varepsilon_i-\varepsilon_j)}
\partial_{\tilde{i}}\partial_{\tilde{j}}(x^{\omega})\in A_{SHO}\cap SHO(\underline{s}
+\varepsilon_i)^{(2)}_{h(SHO)}$;
\item[$\mathrm{(4)}$] $ A_{SKO}\cap \big(SKO(\underline{s}+\varepsilon_i)^{(2)}_{h(SKO)}+SKO(\underline{s}+\varepsilon_i)^{(2)}_{h(SKO)-1}\big)$
contains the element
\begin{eqnarray*}
\begin{split}
 (-1)^{m}(m\lambda+3)x^{(\pi+(p-1)p^{s_i}\varepsilon_i)}\partial_{\tilde{i}}
 (x^{\omega-\langle{2m+1}\rangle})
&-\sum_{j=1}^mx^{(\pi+(p-1)p^{s_i}\varepsilon_i-\varepsilon_j)}
\partial_{\tilde{i}}\partial_{\tilde{j}}(x^\omega),
\end{split}
\end{eqnarray*}
\end{itemize}
where $h(L)$ is the
height  of
$L({\underline{s}+\varepsilon_i})^{(2)}$ (see
Remark \ref{ctr101120}). Now the conclusion follows from   Lemma \ref{ct101120l2.2}.
\end{proof}

\begin{proposition}\label{ct101120l2.3}
Let $A_X$ denote a subalgebra of $X(\underline{\infty})$ and $\underline{s}\geq \underline{1}$.
Put
\[ \begin{array}{lll} E_W=x^{(p^{s_i}\varepsilon_i)}\partial_j\, \,\mbox{for some} \, \, j\in \mathbf{I}_0; &E_H=x^{((p^{s_i}+1)\varepsilon_i)};
\\E_S=D_{ij}(x^{((p^{s_i}+1)\varepsilon_i)})\, \,\mbox{for some} \, \, j\in \mathbf{I}_0\backslash\{i\}; &E_K=x^{(p^{s_i}\varepsilon_i)};
\\E_{HO}=E_{SHO}=x^{((p^{s_i}+1)\varepsilon_i)}; &E_{KO}=E_{SKO}=x^{(p^{s_i}\varepsilon_i)}.
\end{array}\]
If $X(\underline{s})^{(2)}+\mathbb{F}E_{X}\subset A_X$, then $X(\underline{s}+\varepsilon_i)^{(2)}\subset A_X$.
\end{proposition}
\begin{proof}
By Lemma \ref{ct101120l2.2}, it is sufficient to
  show that $A_X\cap
X(\underline{s}+\varepsilon_i)^{(2)}_{h(X)-1}\neq 0$ or $A_X\cap
X(\underline{s}+\varepsilon_i)^{(2)}_{h(X)}\neq 0$, where $h(X)$ is the
height of
$X({\underline{s}+\varepsilon_i})^{(2)}$ (see
Remark \ref{ctr101120}).
From \cite[Lemma 5.2.6]{Strade} it follows that
\begin{eqnarray*}
\begin{split}
 &f_W:=x^{(\pi+(p-1)p^{s_i}\varepsilon_i)}\partial_r\in A_W\cap W(\underline{s}+\varepsilon_i)_{h(W)-n}, \mbox{ where } r\in \mathbf{I}_0\backslash\{i\};\\
 &f_H:=x^{(\pi+[(p-1)p^{s_i}-1]\varepsilon_i-\varepsilon_{i'})} \in A_H\cap
H(\underline{s}+\varepsilon_i)^{(2)}_{h(H)-n-1};\\
&f_K:=x^{(\pi+(p-1)p^{s_i}\varepsilon_i-\varepsilon_\rho)} \in A_K\cap \big(K(\underline{s}+\varepsilon_i)^{(2)}_{h(K)-n}+K(\underline{s}+\varepsilon_i)^{(2)}_{h(K)-n-1}\big),
\end{split}
\end{eqnarray*}
where $\rho=i$ if $i\not=m$; $\rho=1$ if $i=m$.

Computing $ \big[f_W, x_rx^\omega\partial_l]$, $\big[f_H, x^{(2\varepsilon_i)}x^\omega\big]_{H}$ and $\big[f_K, x^{(2\varepsilon_\rho)}x^\omega\big]_K$,
 respectively, one gets
\begin{itemize}
\item[$\mathrm{(1)}$] $x^{(\pi+(p-1)p^{s_i}\varepsilon_i)}x^\omega\partial_l \in A_{W}\cap W(\underline{s}+\varepsilon_i)^{(2)}_{h(W)}$, where
$l\in \mathbf{I}_1$;
\item[$\mathrm{(2)}$] $\sigma(i)x^{(\pi+(p-1)p^{s_i}\varepsilon_i-2\varepsilon_{i'})}x^\omega\in A_{H}\cap H(\underline{s}+\varepsilon_i)^{(2)}_{h(H)-1}$;
\item[$\mathrm{(3)}$] $\sigma(\rho)x^{(\pi+(p-1)p^{s_i}\varepsilon_i-\varepsilon_{\rho'})}x^\omega\in
A_K\cap \big(K(\underline{s}+\varepsilon_i)^{(2)}_{h(K)}+K(\underline{s}+\varepsilon_i)^{(2)}_{h(K)-1}\big)$.
\end{itemize}

It follows that $X(\underline{s}+\varepsilon_i)^{(2)}\subset A_X$ for $X=W, H$ or $K$.
For $X=HO$ or $SHO$, choosing $j\in \mathbf{I}_0\backslash\{i\}$, we have
\begin{itemize}
\item[]$x^{(p^{s_i}\varepsilon_i)}x_{\tilde{j}}=\big[x^{((p^{s_i}+1)\varepsilon_i)},
x_{\tilde{i}}x_{\tilde{j}}\big]_{HO}\in A_{X}.$
\end{itemize}
 Then
$$x^{((p^{s_i+1}-1)\varepsilon_i)}=\big((p-1)!\big)^{-1}\big(\mathrm{ad}x^{(p^{s_i}\varepsilon_i)}x_{\tilde{j}}\big)^{p-1}
\big(x^{((p^{s_i}-1)\varepsilon_i+(p-1)\varepsilon_{j}}\big)\in A_{X}.$$
Observe  that, for $X=KO$ or $SKO$,
\begin{itemize}
\item[]$x^{((p^{s_i}+1)\varepsilon_i)}=-2^{-1}\big[x^{(p^{s_i}\varepsilon_i)},
x_{i}x_{2m+1}\big]_{KO}\in A_{X}.$
\end{itemize}
As in the case
$X=HO$, we  have
$x^{((p^{s_i+1}-1)\varepsilon_i)}\in A_X.$
From Lemma \ref{ct1101211513}, we   obtain that $X(\underline{s}+\varepsilon_i)^{(2)}\subset A_X$
for $X=HO, SHO, KO$ or $ SKO.$

For $X=S$, without loss of   generality, we may
assume inductively  that $A_S$ contains
$B_l:=D_{ij}(x^{(\pi+(l-1)p^{s_i}\varepsilon_i)}x^{\omega})$,
where  $1\leq l\leq p-1$.
Note that
for $1\leq a$, $1\leq b\leq p-1$, $1\leq l\leq p,$
\begin{itemize}
\item[]$\begin{pmatrix}
lp^a-b\\
p^a
\end{pmatrix}
=
\begin{pmatrix}
l-1\\
1
\end{pmatrix}
=l-1.$
\end{itemize}
Then
\begin{itemize}
\item[]$D_{ij}(x^{(\pi+lp^{s_i}\varepsilon_i-\varepsilon_j)}
x^{\omega})=l^{-1}\big[D_{ij}(x^{((p^{s_i}+1)\varepsilon_i)}), B_l\big]\in A_S.$
\end{itemize}
Choose any $k\in\mathbf{I}_1.$ One sees that $A_S$ contains
\begin{itemize}
\item[]$\big[D_{ik}(x^{(\varepsilon_i+\varepsilon_j)}x_k), [D_{ij}(x^{((p^{s_i}+1)\varepsilon_i)}), B_l]\big]=-lB_{l+1},$
\end{itemize}
which implies that $0\neq B_p=D_{ij}(x^{(\pi+(p-1)p^{s_i}\varepsilon_i)}
x^{\omega})\in A_S\cap S(\underline{s}+\varepsilon_i)^{(2)}_{h(S)}$.
The proof is complete.
\end{proof}

%%%%%%%%%%%%%%%%%%%%%%%%%%%%%%%%%%%%%%%%%%%%%%%%%%%%%%%%%%%%%%%%%%%%%%%%%%%%%%%%%%%%%%%%%%%%%%%%%%%%%%%%%%%%%%%%%%%%%%%%%%%%%%%%%%%%%%%%%%%%%%%%%%%%%%%%%%%%%%%%%%%%%%%%%%%%%%%%%%%%%%%%%%%%%%%%
%%%%%%%%%%%%%%%%%%%%%%%%%%%%%%%%%%%%%%%%%%%%%%%%%%%%%%%%%%%%%%%%%%%%%%%%%%%%%%%%%%%%%%%%%%%%%%%%%%%%%%%%%%%%%%%%%%%%%%%%%%%%%%%%%%%%%%%%%%%%%%%%%%%%%%%%%%%%%%%%%%%%%%%%%%%%%%%%%%%%%%%%%%%%
%%%%%%%%%%%%%%%%%%%%%%%%%%%%%%%%%%%%%%%%%%%%%%%%%%%%%%%%%%%%%%%%%%%%%%%%%%%%%%%%%%%%%%%%%%%%%%%%%%%%%%%%%%%%%%%%%%%%%%%%%%%%%%%%%%%%%%%%%%%%%%%%%%%%%%%%%%%%%%%%%%%%%%%%%%%%%%%%%%
\section{Superderivations }
 As before, $X=W$, $S$, $H$, $K$,
$HO$, $SHO$, $KO$ or $SKO$. Apparently, it is much easier
  to   determine the superderivations of $X(\underline{1})^{(2)}$ than to determine
   the superderivations of $X(\underline{t})$.
   On the other hand, $X(\underline{1})^{(2)}$ contains
   almost the
  whole elementary information of $X(\underline{t})$ in a sense: $X(\underline{1})^{(2)}$ contains all the
  formal variables of the underlying superalgebras and all the partial derivatives $\partial_i$ of
   $X(\underline{t})$. As expected, starting from the superderivation algebras of the ``basic'' subalgebra
   $X(\underline{1})^{(2)}$, we are able to determine the superderivations of the ``big'' algebra $X(\underline{t})$ and
   its derived algebra no matter  $X(\underline{t})$ is finite dimensional or not.

 We first introduce two ``exceptional'' superderivations for $HO$ and $SHO.$
\begin{itemize}
\item[$\mathrm{(1)}$]  By  \cite{Hua-Liu},   $HO(\underline{t})$ has an outer derivation
\begin{itemize}
\item[] $\Phi: HO(\underline{t}) \longrightarrow HO(\underline{t}),\quad
 f \longmapsto \sum_{i\in
\mathbf{I}_0}\partial_i\partial_{\tilde{i}}(f).$
\end{itemize}
\item[$\mathrm{(2)}$] By a direct computation,   we can show that $SHO(3,3;\underline{t})^{(2)}$ has an outer derivation
\begin{itemize}
\item[]
$\Theta: SHO(\underline{t}) \longrightarrow SHO(\underline{t}), \quad f \longmapsto \tau(f), $
\end{itemize}
where $\tau: \mathcal{O}(\underline{t}) \longrightarrow \mathcal{O}(\underline{t})$ is a linear operator such that for
$\alpha=\alpha_1\varepsilon_1+\alpha_2\varepsilon_2+\alpha_3\varepsilon_3, $
\begin{itemize}
\item[] $\tau(x^{(\alpha)}x^u)=(a_1x^{(\varepsilon_1)}\partial_{\tilde{2}}\partial_{\tilde{3}}
+a_2x^{(\varepsilon_2)}\partial_{\tilde{3}}\partial_{\tilde{1}}
+a_3x^{(\varepsilon_3)}\partial_{\tilde{1}}\partial_{\tilde{2}})(x^{(\alpha)}x^u),$
\end{itemize}
where
\begin{itemize}
\item[] $a_i= \left\{\begin{array}{ll}
\big[(1+b_i)(\alpha_i+1)\big]^{-1} \quad &\alpha_i\not\equiv -1\pmod{p},\\
0 &\alpha_i\equiv -1\pmod{p},
\end{array}\right.$
\end{itemize}
where $b_i=\sum_{j\in \{1, 2, 3\}}\delta_{\alpha_j\neq 0}\delta_{\tilde{j}\in u}$.
\end{itemize}

\begin{remark}\label{ct110611rn1}
In \cite{BLL},   $\Theta$  mentioned in (2) is neglected by mistake when $m=3$.
\end{remark}
As in Lie algebra case,   $X(\underline{1})^{(2)}$ is generated by its local part and one may determine its superderivations by a direct computation. Here we single out certain conclusions on the negative superderivations from \cite{BLL,Fu-Zhang-Jiang,Ma-Zhang,Liu-Zhang-Wang,Wang-Zhang,Zhang-Zhang}:

\begin{remark}\label{ct101120r2.1}
\begin{itemize}
\item[]$\mathrm{Der}^-(X(\underline{1})^{(2)})\cong X^{-}(\underline{1})^{(2)}
\oplus\delta_{X=HO}\mathbb{F}\Phi\oplus\delta_{X=SHO}\delta_{m=3}\mathbb{F}\Theta.$
\end{itemize}
Furthermore,
$$\big\{\phi \in\mathrm{Der}(X(\underline{1})^{(2)})\mid \phi(X^-(\underline{1})^{(2)})=0 \big\}\cong X(\underline{1})^{(2)}_{-r}
\oplus\delta_{X=HO}\mathbb{F}\Phi\oplus\delta_{X=SHO}\delta_{m=3}\mathbb{F}\Theta,
$$
where $r$ is the depth of $X(\underline{1})^{(2)}$.
\end{remark}

 Note that $T:=\sum_{i\in \mathbf{I}}\mathbb{F}x_i\partial_i$ is abelian and acts diagonally on  $W(m,n)$.  We call  $T_X:=X\cap T$ the \textit{canonical torus} of $X$.
 Following \cite[Lemma 6.1.3 ]{Strade}, we have

\begin{lemma}\label{101123ctt2.1}
Let $\widetilde{X}$ be a $\mathbb{Z}$-graded subalgebra of $X(\underline{t})$ containing  $X(\underline{t})^{(2)}$
and ${Q}:=\Big\{\phi \in\mathrm{Der}\big(\widetilde{X}, X(\underline{\infty})\big)\mid \phi(\widetilde{X}^-)=0 \Big\}$. Then
\begin{itemize}
\item[]$Q\cong \widetilde{X}_{-r} \oplus
\mathrm{span}_{\mathbb{F}}\Big\{
\sum_{j=1}^{\infty}\mathbb{F}\partial_i^{p^{j}}\mid i\in\mathbf{I}_0\Big\}
\oplus\delta_{X=HO}\mathbb{F}\Phi\oplus\delta_{\widetilde{X}=SHO(3,3;\underline{t})^{(2)}}\mathbb{F}\Theta,$
\end{itemize}
where $r$ is the depth of $\widetilde{X}$.
\end{lemma}
\begin{proof}
For any $\phi\in Q$, we can consider the following cases:\\

\noindent\textbf{Case 1:} $\underline{t}\neq\underline{\infty}$.
From Proposition \ref{ct11012110582.2} we know that $\phi$ leaves $X(\underline{1})^{(2)}$ invariant. In view of Remark
\ref{ct101120r2.1}  we may assume that $X(\underline{1})^{(2)}\subset
\ker \phi.$

Suppose  $\underline{1}\leq\underline{s}\leq\underline{t}$ to be maximal element satisfying
$X(\underline{s})^{(2)}\subset \ker \phi.$ Then
\begin{itemize}
\item[]$\Big[X(\underline{1})^{(2)},
\phi\big(\widetilde{X}\cap X(\underline{s})\big)\Big] \subset
\phi\big(X(\underline{s})^{(2)}\big)=0.$
\end{itemize}
In addition, for $X=H, HO$ or $SHO$,
\begin{itemize}
\item[]$\Big[X(\underline{1})^{(2)},
\phi\big(\widetilde{X}\cap \overline{X}(\underline{s})\big)\Big] \subset
\phi\big(X(\underline{s})^{(2)}\big)=0.$
\end{itemize}
Whence $\phi\big(\widetilde{X}\cap
X(\underline{s})\big)=0$ and $\phi\big(\widetilde{X}\cap
\overline{X}(\underline{s})\big)=0$.
The conclusion holds if $\underline{s}=\underline{t}$.
Suppose  $\underline{s}<\underline{t}$ and let ${i_0}$ be an index such that
 $s_{i_0}<t_{i_0}$.
Fix any $j \in \mathbf{I}_0$ and consider the   elements $E_X\in \widetilde{X}$ listed below:
\begin{itemize}
\item[] $\begin{array}{ll}
 E_W=x^{(p^{s_{i_0}}\varepsilon_{i_0})}\partial_j; &E_H=x^{((p^{s_{i_0}}+1)\varepsilon_{i_0})};
\\E_S=D_{i_0j_0}(x^{((p^{s_{i_0}}+1)\varepsilon_{i_0})}); &E_K=x^{(p^{s_{i_0}}\varepsilon_{i_0})};
\\E_{HO}=E_{SHO}=x^{((p^{s_{i_0}}+1)\varepsilon_{i_0})}; &E_{KO}=E_{SKO}=x^{(p^{s_{i_0}}\varepsilon_{i_0})}.
\end{array}$
\end{itemize}
By Proposition \ref{ct101120l2.3}, $E_X \not\in\ker \phi.$
However, a computation shows that
$\left[E_X, X^{-}(\underline{s})\right]\subset\ker \phi,$
whence
\begin{equation}\label{1102121432}\Big[\phi(E_X), X^{-}(\underline{s})\Big]=0.
\end{equation}

Let $T_X$ be the canonical torus of $X$.
By Lemma \ref{obl201010171415}  we  can assume that $\phi$ is a
zero weight-superderivations. Since
$X(\underline{t})$ is centerless, we have  $\phi(T_X)=0$.
 Thus we can obtain the following results.\\

\noindent\textit{Case 1:} Suppose $X=W$, $S$, $H$, $HO$ or $SHO$. (\ref{1102121432})
means $\phi(E_X) \in \sum_{k \in \mathbf{I}}\mathbb{F}\partial_k.$
By a direct computation, we can find  $\beta_{X} \in \mathbb{F}$
such that
\begin{itemize}
\item[] $\begin{array}{lll}
\phi(E_W)=\beta_W\partial_{i_0};&\phi(E_S)=\beta_S\partial_{j_0};& \phi(E_H)=\beta_{H}x_{i_0};\\
\phi(E_{HO})=\beta_{HO}x_{i_0};& \phi(E_{SHO})=\beta_{SHO}x_{i_0}.
 \end{array}$

\end{itemize}
 Then
$\phi-\beta_X\mathrm{ad}\partial_{i_0}^{p^{s_{i_0}}}$ vanishes on
$X(\underline{s})^{(2)}+\mathbb{F}E_X.$\\

\noindent\textit{Case 2:} Suppose $X=K$. From (\ref{1102121432}) we have $\phi(E_K)\in K(\underline{s})_{-2}$ and therefore,   $\phi(E_K)=2\beta_K\partial_m$ for some $\beta_K\in\mathbb{F}$.
Then $\phi-\beta_K\mathrm{ad}\partial_{i_0}^{p^{s_{i_0}}}$ vanishes on
$K(\underline{s})^{(1)}+\mathbb{F}E_K.$\\

\noindent\textit{Case 3:} Suppose $X=KO$ or $SKO$. From (\ref{1102121432})  we have $\phi(E_X)\in X(\underline{s})_{-2}$ and $\phi(E_X)=-2\beta_X\partial_{2m+1}$ for some  $\beta_X\in\mathbb{F}$. Then
$\phi-\beta_X\mathrm{ad}\partial_{i_0}^{p^{s_{i_0}}}$ vanishes on
$X(\underline{s})^{(2)}+\mathbb{F}E_X.$

Summarizing, one sees from Proposition   \ref{ct101120l2.3}   that $\phi$   vanishes on
$X(\underline{s}+\varepsilon_{i_0})^{(2)}$ modulo  $\beta_X\mathrm{ad}\partial_{i_0}^{p^{s_{i_0}}}$.
By induction on $\underline{s}$, we may assume that $X(\underline{t})^{(2)} \subset \ker \phi.$
It follows that $\phi(\widetilde{X})=0 $ and the proof  in this case is complete.\\

\noindent\textbf{Case 2:} $\underline{t}=\underline{\infty}$.
From
 Proposition \ref{ct11012110582.2} we
know that $\phi $ leaves  $X(\underline{t})^{(2)}$ invariant  for any $\underline{t}\not=
\underline{\infty}$. From Case 1  we can assume that
 \begin{itemize}
\item[]
$\phi|_{X(\underline{t})^{(2)}}= \mathrm{ad}D_t+
\sum_{i=1}^m\sum_{j=1}^{\infty}a(t)_{ij}\mathrm{ad}\partial_i^{p^j}
+\mu_{t}\delta_{X=HO}\Phi+\nu_{t}\delta_{X=SHO}\delta_{m=3}\Theta,$
 \end{itemize}
where $D_t\in X(\underline{t})^{(2)}_{-r} $
and $r$ is the depth of $X(\underline{t})^{(2)}$  and $a(t)_{ij}, \mu_{t}, \nu_{t}\in\mathbb{F}$.
 A direct computation shows that
 \begin{itemize}
\item[]
 $a(t)_{ij}=a(s)_{ij}\, \,   \mbox{ for } \underline{t}\leq \underline{s};$
\item[]
$\mu_{t}=\mu_{s}, \quad \nu_{t}=\nu_{s}, \quad D_t=D_s\,  \, \mbox{ for all } \, \underline{t},
\underline{s}\in \mathbb{N}_0^m.$
 \end{itemize}
Put $D_{\phi}:=D_t$, $\mu_{\phi}:=\mu_t$, $\nu_{\phi}:=\nu_t$ and
$\varphi:=\phi-\mathrm{ad}D_{\phi}-\mu_{\phi}\delta_{X=HO}\Phi-\nu_{\phi}\delta_{X=SHO}\delta_{m=3}\Theta$.
Then
 \begin{itemize}
\item[]
$\varphi|_{X(\underline{t})^{(2)}}=
\sum_{i=1}^m\sum_{j=1}^{\infty}a(t)_{ij}\partial_i^{p^j}.$
 \end{itemize}
%Hence we can suppose $b=b^{t}$, $c=c^{t}$ and
%$$\phi|_{X(\underline{t})^{(2)}}=
%\sum_{k\in\mathbf{I}_0}\sum_{\gamma(k)=1}^{t_k-1}a^{t}_{k, \gamma(k)}(\mathrm{ad}\partial_k)^{p^{\gamma(k)}}
%+b\delta_{X=HO}\Phi+c\delta_{X=SHO}\delta_{m=3}\Theta.$$
For $i\in \overline{1, m}$ and $j>0$, choose $\lambda_{ij}^{X}\in \mathbb{F}$ such that
$\lambda_{ij}^{X}$ is the coefficient of $\frak{B}_X$ in
$\varphi(\frak{C}_X)$, where $\frak{B}_X$, $\frak{C}_X\in X(\underline{t})^{(2)}$.
Further information is listed below:

~~~~~\begin{tabular}{|l|l|l|}
 \multicolumn{3}{c}{}
    \\[5pt] \hline
     \multicolumn{1}{|c|}{$\lambda_{ij}^{X}$}&
     \multicolumn{1}{|c|}{$\frak{B}_X$}&
     \multicolumn{1}{|c|}{$\frak{C}_X$}\\
     \hline
      $\lambda_{ij}^{W}$&
     $x^{\omega}\partial_l$ &
     $x^{(p^{j}\varepsilon_i)}x^{\omega}\partial_l$
     \\
     \hline
      $\lambda_{ij}^{S}$&
     $D_{il}(x^{\omega})$ &
     $D_{il}(x^{(p^{j}\varepsilon_i)}x^{\omega})$
     \\
     \hline
      $\lambda_{ij}^{H}$&
     $x^{\omega}$ &
     $x^{(p^{j}\varepsilon_i)}x^{\omega}$
     \\
      \hline
      $\lambda_{ij}^{K}$&
     $1$ &
     $x^{(p^{j}\varepsilon_i)}$
     \\
     \hline
      $\lambda_{ij}^{HO}$&
     $x^{\omega}$ &
     $x^{(p^{j}\varepsilon_i)}x^{\omega}$
     \\
      \hline
      $\lambda_{ij}^{SHO}$&
     $x^{\omega-\langle{l}\rangle}$ &
     $x^{(p^{j}\varepsilon_i)}x^{\omega-\langle{l}\rangle}$, where $l\not=\tilde{i}$
     \\
      \hline
      $\lambda_{ij}^{KO}$&
     $x_{2m+1}$ &
     $x^{(p^{j}\varepsilon_i)}x_{2m+1}$
     \\
      \hline
      $\lambda_{ij}^{SKO}$&
     $x_{2m+1}+m\lambda(x_ix_{\tilde{i}})$ &
     $x^{(p^{j}\varepsilon_i)}x_{2m+1}+m\lambda(x^{((p^{j}+1)\varepsilon_i)}x_{\tilde{i}})$
     \\\hline
   \end{tabular}\\

Put
 \begin{itemize}
\item[]$\delta^X=\sum_{i=1}^m\sum_{j=1}^{\infty}\lambda_{ij}^X\mathrm{ad}\partial_i^{p^j}$.
 \end{itemize}
Clearly, $(\varphi-{\delta^X})(X(\underline{t})^{(2)})=0$.
It follows that $(\varphi-{\delta^X})(\widetilde{X})=0 $ and the proof  in this case is complete.

\end{proof}

By a  computation we are able to show that
 \begin{itemize}
 \item[] $\mathrm{Nor}_{W(\underline{t})}X(\underline{t})
=\mathrm{Nor}_{W(\underline{t})}X(\underline{t})^{(1)}$ for $X=S, H$ or $K$;
\item[] $\mathrm{Nor}_{W(\underline{t})}X(\underline{t})
=\mathrm{Nor}_{W(\underline{t})}X(\underline{t})^{(1)}=
\mathrm{Nor}_{W(\underline{t})}X(\underline{t})^{(2)}$ for $X=SHO$ or $SKO$.
\end{itemize}
 Further information is listed below (c.f. \cite{BLL,Fu-Zhang-Jiang,Liu-Zhang-Wang,Liu-Yuan-1,Ma-Zhang,Wang-Zhang,Zhang-Zhang}):\\

\begin{tabular}{|l|l|l|l|l|l|l|l|l|}
 \multicolumn{8}{c}{Table 3.1: Normalizer of $X(\underline{t}) $ in $W(\underline{t})$}
    \\[5pt] \hline
     \multicolumn{1}{|c|}{$X$}&
     \multicolumn{1}{|c|}{$S$}&
     \multicolumn{1}{|c|}{$H$}&
     \multicolumn{1}{|c|}{$K$}&
     \multicolumn{1}{|c|}{$HO$}&
     \multicolumn{1}{|c|}{$SHO$}&
     \multicolumn{1}{|c|}{$KO$}&
     \multicolumn{1}{|c|}{$SKO$}
     \\ \hline
     {\small Nor}&
     $\overline{S}$ &
     $ \overline{H} \oplus\mathbb{F}\mathfrak{D}$ &
     $K$&
     $\overline{HO}\oplus\mathbb{F}\mathfrak{D}$&
     $\overline{SHO}\oplus\mathbb{F}\mathfrak{D}$&
     $KO$&
     $SKO\oplus \mathbb{F}x_{1}x_{\tilde{1}}$
  \\\hline
   \end{tabular}\\

\noindent Hereafter,
\begin{itemize}
\item[]
 $\mathfrak{D}=\sum_{i\in \mathbf{I}}x_{i}\partial_{i}$,
 the degree derivation of
$X(\underline{t})$, where $X=S, H, HO$ or $SHO$.
 \end{itemize}

Let $M_{m\times \infty}$ be the vector space of all $m\times \infty$ matrices over $\mathbb{F}$, that is
\begin{itemize}
\item[]$M_{m\times \infty}:=\left
\{\sum_{i=1}^m\sum_{j=1}^{\infty}\mathbb{F}e_{ij} \mid e_{ij} \mbox{ is the unit of  } m\times \infty \mbox{ matrix }\right\},$
\end{itemize}
which is regarded as  an abelian
subalgebra of $\mathrm{Der}\big(L(\underline{t})\big)$ by letting
\begin{itemize}
\item[]
$[e_{ij}, D]=[\partial_i^{p^j}, D],$ for any $D\in W(\underline{t})$,
 \end{itemize}
where $L=X$, $X^{(1)}$, or $X^{(2)}.$
%Put $P:=\sum_{i=1}^{m}\sum_{j=1}^{\infty}\mathbb{F}\mathrm{ad}\partial_i^{p^j}$ which is an abelian subalgebra of $\mathrm{Der}\big(X(\underline{t})\big)$.
%Obviously,
%$$A=\begin{pmatrix}\lambda_{11} & \lambda_{12} &  \cdots \\ \vdots & \vdots & \vdots \\
%\lambda_{m1} & \lambda_{m2} &  \cdots\end{pmatrix}\in M_{m\times \infty} \mbox{ and } \delta_A=\sum_{i=1}^m\sum_{j=1}^{\infty}\lambda_{ij}\mathrm{ad}\partial_i^{p^{j}}\in P$$
%is one-to-one. Thus $M_{m\times \infty}\cong P$.

Put
 $[e_{ij}, \Phi]=[e_{ij}, \Theta]=0.$
\begin{theorem}\label{101123ctc1}
\begin{equation*}\begin{split}
\mathrm{Der}\big(L(\underline{t})\big)\cong\left(\mathrm{Nor}_{W(\underline{t})}
L(\underline{t})\right)\oplus
\delta_{X=HO}\mathbb{F}\Phi\oplus\delta_{L=SHO^{(2)}}\delta_{m=3}\mathbb{F}\Theta
\oplus  M_{m\times \infty},
\end{split}
\end{equation*}
where $L=X$, $X^{(1)}$, or $X^{(2)}.$
\end{theorem}
\begin{proof}
It is a direct result of  Proposition \ref{ct101120l2.1} and  Lemma
\ref{101123ctt2.1}.

\end{proof}

%%%%%%%%%%%%%%%%%%%%%%%%%%%%%%%%%%%%%%%%%%%%%%%%%%%%%%%%%%%%%%%%%%%%%%%%%%%%%%%%%%%%%%%%%%%%%%%%%%%%%%%%%%%%%%%%%%%%%%%%%%%%%%%%%%%%%%%%%%%%%%%%%%%%%%%%%%%%%%%%%%%%%%%%%%%%%%%%%%%%%%%%%%%%%%%%
%%%%%%%%%%%%%%%%%%%%%%%%%%%%%%%%%%%%%%%%%%%%%%%%%%%%%%%%%%%%%%%%%%%%%%%%%%%%%%%%%%%%%%%%%%%%%%%%%%%%%%%%%%%%%%%%%%%%%%%%%%%%%%%%%%%%%%%%%%%%%%%%%%%%%%%%%%%%%%%%%%%%%%%%%%%%%%%%%%%%%%%%%%%%
%%%%%%%%%%%%%%%%%%%%%%%%%%%%%%%%%%%%%%%%%%%%%%%%%%%%%%%%%%%%%%%%%%%%%%%%%%%%%%%%%%%%%%%%%%%%%%%%%%%%%%%%%%%%%%%%%%%%%%%%%%%%%%%%%%%%%%%%%%%%%%%%%%%%%%%%%%%%%%%%%%%%%%%%%%%%%%%%%%

\section{ Outer superderivations }
In this section, let $X=W$, $S$, $H$, $K$, $HO$, $SHO$, $KO$ or
$SKO$ and $L=X(\underline{t})$, $X(\underline{t})^{(1)}$ or $X(\underline{t})^{(2)}$. Denote by $\mathrm{Der_{out}}(L):=\mathrm{Der}(L)/\mathrm{ad}(L)$
 the outer superderivation algebra of $L$.  Write $\delta'_{i,j}=1$ if $i\equiv j\pmod{p}$; $\delta'_{i, j}=0$ otherwise.

For future reference, we establish  the following Lie algebras.
\begin{itemize}
\item[$\mathrm{(1)}$] Put $\frak{G}^{X}(\underline{t}):=M_{m\times \infty}$
 when $X=W, K$ or $KO$. For $\underline{t}\not=\underline{\infty}$, the direct sum of Lie algebras
  \begin{itemize}
\item[]${\frak{G}}_1^{K}(\underline{t}):=\frak{G}^{K}(\underline{t})\oplus \delta'_{m-n, 3}\mathbb{F}f^{K}$
 \end{itemize}
is an abelian Lie algebra.

\item[$\mathrm{(2)}$] Write the direct sum of Lie algebras ${\frak{G}}^{S}(\underline{t}):=M_{m\times \infty}\oplus \mathbb{F}f^S$
and let $V^S:=\mathrm{span}_{\mathbb{F}}\{g_i^S\mid i\in \overline{1, m}\}$ be  an abelian Lie algebra.
Then, for $\underline{t}\not=\underline{\infty}$, the semi-direct sum
 \begin{itemize}
\item[]${\frak{G}}^{S}_1(\underline{t}):={\frak{G}}^{S}(\underline{t})\ltimes_{\mathrm{ad}} V^S$
 \end{itemize}
is a  Lie algebra
with multiplication $[M_{m\times \infty}, V^S]=0$ and  $ [f^S, g_i^S]=g_i^S.$
\item[$\mathrm{(3)}$] Write the direct sum of Lie algebras  $\overline{\frak{G}}^{H}(\underline{t}):=M_{m\times \infty}\oplus \mathbb{F}f_1^H$
 and let $V^H:=\mathrm{span}_{\mathbb{F}}\{g_i^H\mid i\in \overline{1, m}\}$ be an  abelian Lie algebras.
Then, the semi-direct sum
 \begin{itemize}
\item[]${\frak{G}}^{H}(\underline{t}):=\overline{\frak{G}}^{H}(\underline{t})\ltimes_{\mathrm{ad}} \delta_{\underline{t}\not=\underline{\infty}}V^H$
 \end{itemize}
 is a Lie algebra
with multiplication $[M_{m\times \infty}, V^H]=0$ and $ [f^H, g_i^H]=-2g_i^H.$ Moreover, for   $\underline{t}\not=\underline{\infty}$,
 the semi-direct sum
 \begin{itemize}
\item[]${\frak{G}}_1^{H}(\underline{t}):=\frak{G}^{H}(\underline{t})\ltimes_{\mathrm{ad}} \mathbb{F}f_2^{H}$
 \end{itemize}
is a  Lie algebra
with multiplication
$[M_{m\times \infty}, f_2^H]=[V^H, f_2^H]=0$ and $ [f_1^H, f_2^H]=(n-m-2)f_2^H$.
\item[$\mathrm{(4)}$]Let $\overline{V}^{HO}:=\mathrm{span}_{\mathbb{F}}\{f_1^{HO}, f_2^{HO}\}$ be a Lie algebra given by
$[f_1^{HO}, f_2^{HO}]=-2f_2^{HO}$.
Then the direct sum of Lie algebras
$\overline{\frak{G}}^{HO}(\underline{t}):=M_{m\times \infty}\oplus \overline{V}^{HO}$
is a Lie algebra.
 Let $V^{HO}:=\mathrm{span}_{\mathbb{F}}\{g_i^{HO}\mid i\in \overline{1, m}\}$ be an abelian Lie algebra.
   Then, the semi-direct sum
    \begin{itemize}
\item[]${\frak{G}}^{HO}(\underline{t}):=\overline{\frak{G}}^{HO}(\underline{t})\ltimes_{\mathrm{ad}} \delta_{\underline{t}\not=\underline{\infty}}V^{HO}$
 \end{itemize}
is a Lie algebra
with multiplication  $[M_{m\times \infty}, V^{HO}]=[f_2^{HO}, V^{HO}]=0$ and $[f_1^{HO}, g_i^{HO}]=-2g_i^{HO}$.
\end{itemize}
\begin{theorem}\label{ctt1102241908} Let $X=W$, $S$, $H$, $K$, $HO$ or $KO$.
The outer superderivation algebras are as follows:
\begin{itemize}
\item[]
$\mathrm{Der_{out}}\big(X(\underline{t})\big)\cong {\frak{G}}^{X}(\underline{t});$
\item[]
$\mathrm{Der_{out}}\big(X(\underline{t})^{(1)}\big)\cong {\frak{G}}_1^{X}(\underline{t}), \mbox{ where } X=S, H \mbox{ or } K,$
and $\underline{t}\not=\underline{\infty}$.
\end{itemize}
In particular, they are all Lie algebras. Furthermore, $\mathrm{Der_{out}}\big(K(\underline{t})^{(1)}\big)$, $\mathrm{Der_{out}}\big(H(\underline{\infty})\big)$ and $\mathrm{Der_{out}}\big(X(\underline{t})\big)$ are  abelian when $X=W, S, H, K$ or $ KO$.
\end{theorem}
\begin{proof}
From \cite{Fu-Zhang-Jiang,Liu-Zhang-Wang,Ma-Zhang,Wang-Zhang,Zhang-Zhang} we have
\begin{itemize}
\item[$\mathrm{(1)}$]
$\overline{S}(\underline{t})=S(\underline{t})\oplus\mathbb{F}\frak{D}$

\quad\quad\,$=S(\underline{t})^{(1)}\oplus\delta_{\underline{t}\not=\underline{\infty}}\sum_{i\in \mathbf{I}_0}\mathbb{F}x^{(\pi-\pi_i\varepsilon_i)}x^\omega\partial_i\oplus\mathbb{F}\frak{D}$;
\item[$\mathrm{(2)}$]
$\overline{H}(\underline{t})=H(\underline{t})\oplus\delta_{\underline{t}\not=\underline{\infty}}\sum_{i\in \mathbf{I}_0}\mathbb{F}x^{(\pi_i\varepsilon_i)}\partial_{i'}$

\quad\quad\,$=H(\underline{t})^{(1)}\oplus\delta_{\underline{t}\not=\underline{\infty}}\big(\mathbb{F}x^{\pi}x^\omega\oplus\sum_{i\in \mathbf{I}_0}\mathbb{F}x^{(\pi_i\varepsilon_i)}\partial_{\tilde{i}}\big)$;
\item[$\mathrm{(3)}$]
$K(\underline{t})=K(\underline{t})^{(1)}\oplus\delta_{\underline{t}\not=\underline{\infty}}\delta'_{n-m,3}
\mathbb{F}x^{(\pi)}x^\omega$;
\item[$\mathrm{(4)}$]
$\overline{HO}(\underline{t})=HO(\underline{t})\oplus\delta_{\underline{t}\not=\underline{\infty}}\sum_{i\in \mathbf{I}_0}\mathbb{F}x^{(\pi_i\varepsilon_i)}\partial_{\tilde{i}}$.
\end{itemize}
From Table 3.1 and Theorem \ref{101123ctc1}, by a direct computation we have the desired results.
\end{proof}

The structures of $SHO(\underline{t})$ and $SKO(\underline{t})$ are very complicated. Now we consider their outer superderivations.
\begin{itemize}
\item[$\mathrm{(1)}$]Let $\overline{V}^{SHO}:=\mathrm{span}_{\mathbb{F}}\{f_1^{SHO}, f_2^{SHO}\}$ be an abelian Lie algebra and $V^{SHO}:=\mathrm{span}_{\mathbb{F}}\{g_i^{SHO}\mid i\in \overline{1, m}\}$
be  an abelian  Lie superalgebra with $V^{SHO}_{\bar{0}}=0$.
Clearly, the direct sum of Lie algebra
$\overline{\frak{G}}^{SHO}(\underline{t}):=M_{m\times \infty}\oplus\overline{V}^{SHO}$
 is an abelian  Lie algebra. Then the semi-direct sum
  \begin{itemize}
\item[]${\frak{G}}^{SHO}(\underline{t}):=\overline{\frak{G}}^{SHO}(\underline{t})\ltimes_{\mathrm{ad}} \delta_{\underline{t}\not=\underline{\infty}}V^{SHO}$
 \end{itemize}
  is a  Lie (super)algebra with multiplication  $[M_{m\times \infty}, V^{SHO}]=[f_2^{SHO}, V^{SHO}]=0$ and $[f_1^{SHO}, g_i^{SHO}]=-2g_i^{SHO}$.

For $\underline{t}=\underline{\infty}$, the semi-direct sum
 \begin{itemize}
\item[]$\overline{\frak{G}}_1^{SHO}(\underline{t}):={\frak{G}}^{SHO}(\underline{t})\ltimes_{\mathrm{ad}} \mathbb{F}f_3^{SHO}$
 \end{itemize}
is a Lie algebra
with multiplication
$[M_{m\times \infty}, f_3^{SHO}]=0$; $[f_1^{SHO}, f_3^{SHO}]=(m-2)f_3^{SHO}$; $[f_2^{SHO}, f_3^{SHO}]=f_3^{SHO}.$

For $\underline{t}\not=\underline{\infty}$, let $\overline{\Lambda(m)}:=\Lambda(m)$ be an abelian Lie superalgebra by letting
$\overline{\Lambda(m)}_{\bar{i}}:=\Lambda(m)_{\overline{{i+1}}}$, $i=0, 1.$ Then the semi-direct sum
 \begin{itemize}
\item[]${\frak{G}}_1^{SHO}(\underline{t}):={\frak{G}}^{SHO}(\underline{t})\ltimes_{\mathrm{ad}} \overline{\Lambda(m)}$
 \end{itemize} is
 a  Lie superalgebra, having a $\mathbb{Z}_2$-grading
structure   induced by
${\frak{G}}^{SHO}(\underline{t})$ and $\overline{\Lambda(m)}$,
with multiplication
\begin{itemize}
\item[]
$[f_1^{SHO}, x^u]=(2|u|-m-2)x^u;\quad [f_2^{SHO}, x^u]=x^u;\quad [x^u, x^v]=0;$
\item[]$[M_{m\times \infty}, x^u]=0;\quad[g_i^{SHO}, x^u]=(-1)^{(i, u)}\delta_{\tilde{i}\in u}x^{u-\langle \tilde{i}\rangle},$
 \end{itemize}
for all $x^u, x^v\in \overline{\Lambda(m)},$ where $(-1)^{(i, u)}$ is
determined by the equation $\partial_{i'}(x^u)=(-1)^{(i,
u)}x^{u-\langle i'\rangle}.$

Let $V_1^{SHO}=\mathrm{span}_{\mathbb{F}}\big\{f_4^{SHO}, \delta_{m=3}f_5^{SHO}\big\}$ and
 \begin{itemize}
\item[]${\frak{G}}_2^{SHO}(\underline{t}):=\frak{G}_1^{SHO}(\underline{t})\oplus V_1^{SHO}$
 \end{itemize}
 be a $\mathbb{Z}_2$-grading space with
\begin{itemize}
\item[]
$\big({\frak{G}}_2^{SHO}(\underline{t})\big)_{\bar{0}}=\big({\frak{G}}_1^{SHO}(\underline{t})\big)_{\bar{0}}\oplus V_1^{SHO}$; $\big({\frak{G}}_2^{SHO}(\underline{t})\big)_{\bar{1}}=\big({\frak{G}}_1^{SHO}(\underline{t})\big)_{\bar{1}}.$
\end{itemize}
Then ${\frak{G}}_2^{SHO}(\underline{t})$ is a Lie superalgebra by letting
 \begin{itemize}
\item[]$[f_1^{SHO}, f_4^{SHO}]=-4f_4^{SHO};\quad [f_2^{SHO}, f_4^{SHO}]=2f_4^{SHO};$
\item[]$[M_{m\times \infty}, f_4^{SHO}]=[V^{SHO}, f_4^{SHO}]=[\overline{\Lambda(m)}, f_4^{SHO}]=0;$
\item[]$[f_1^{SHO}, f_5^{SHO}]=[f_2^{SHO}, f_5^{SHO}]=-f_5^{SHO};$
\item[]$[M_{m\times \infty}, f_5^{SHO}]=[V^{SHO}, f_5^{SHO}]=0;$
\item[]$[\partial_{j'}(x^\omega), f_5^{SHO}]=g_j^{SHO}$,  $[x_{j'}, f_5^{SHO}]=0;$ { for  } $j=1,2,3;$
\item[]$[1, f_5^{SHO}]=0;$ $ [x^\omega, f_5^{SHO}]=2^{-1}(3f_2^{SHO}+f_1^{SHO});$ $ [f_4^{SHO}, f_5^{SHO}]=1,$
 \end{itemize}

\item[$\mathrm{(2)}$]
 Write the direct sum of Lie algebras ${\frak{G}}^{SKO}(\underline{t}):=M_{m\times \infty}\oplus \mathbb{F}f^{SKO}$.

For $\underline{t}=\underline{\infty}$,
the semi-direct sum
\[\overline{\frak{G}}_1^{SKO}(\underline{t})= \left\{\begin{array}{ll}
{\frak{G}}^{SKO}(\underline{t})\ltimes_{\mathrm{ad}} \mathbb{F}f_1^{SKO}&m\lambda-m+2\equiv 0\pmod{p} \mbox{ or } \lambda=1;\\
{\frak{G}}^{SKO}(\underline{t})&\mbox{otherwise}.
\end{array}\right.\]
 is a Lie algebra with multiplication
$
[M_{m\times \infty}, f_1^{SKO}]=0$ and $ [f^{SKO}, f_1^{SKO}]=f_1^{SKO}.
$

For $\underline{t}\not=\underline{\infty}$,
 we introduce some symbols for simplicity. Let
$(i_1,i_2,\ldots,i_k)$ be a $k$-tuple of pairwise distinct positive
integers.
%As in usual, put
%$$\mathrm{sgn}(i_1,i_2,\ldots,i_k):=\frac{\prod_{1\leq{j}<l\leq k}(i_{l}-i_{j} )}
%{|\prod_{1\leq{j}<l\leq k}(i_{l}-i_{j} )|}.$$
 Write
the  integer
\begin{itemize}
\item[]$l(\lambda, m):=\sum_{k\in \mathfrak{S}_{0}(\lambda,
m)}{m\choose k}+\sum_{k\in
\mathfrak{S}_{2}(\lambda,
m)}{m\choose k},$
\end{itemize}
where
$
\mathfrak{S}_{l}(\lambda, m):=\{k\in \overline{0,m}\mid
m\lambda-m+2k+l=0 \in\mathbb{F}\}.
$
Let $V^{SKO}:=V^{SKO}_{\bar{0}}\oplus V^{SKO}_{\bar{1}}$ be a $\mathbb{Z}_{2}$-graded
vector space where
 \begin{equation*}\begin{split}
&V^{SKO}_{\bar{0}}:= V_{01}\oplus V_{02},\quad V^{SKO}_{\bar{1}}:=V_{11}\oplus V_{12};\\
&V_{01}:=\mathrm{span}_{\mathbb{F}}\{X_{i_1,\ldots,i_r}\mid
r\in\mathfrak{S}_{2}(\lambda, m),(i_{1}, \ldots, i_{r})\in
\mathbf{J}(r), m-r\;\mbox{ is odd}\};\\
&V_{02}:=\mathrm{span}_{\mathbb{F}}\{Y_{j_1,\ldots,j_l}~\mid
 l\in\mathfrak{S}_{0}(\lambda,
m),(j_{1}, \ldots, j_{l})\in \mathbf{J}(l), m-l\;\mbox{ is even}\};\\
&V_{11}:=\mathrm{span}_{\mathbb{F}}\{X_{i_1,\ldots,i_r}\mid
r\in\mathfrak{S}_{2}(\lambda, m),(i_{1}, \ldots, i_{r})\in
\mathbf{J}(r), m-r\;\mbox{ is even}\};\\
&V_{12}:=\mathrm{span}_{\mathbb{F}}\{Y_{j_1,\ldots,j_l}~\mid
l\in\mathfrak{S}_{0}(\lambda,
m),\,(j_{1}, \ldots, j_{l})\in \mathbf{J}(l), m-l\;\mbox{ is odd}\};\\
&\mathbf{J}(0):=\emptyset;\quad
\mathbf{J}(r):=\{(i_{1}, \ldots, i_{r})\mid  1\leq i_{1}<\cdots< i_{r}  \leq m\}.
\end{split}
\end{equation*}
Moreover, $V^{SKO}$ is a Lie superalgebra by letting
\begin{itemize}
\item[]$[V_{01}+V_{11},V_{01}+V_{11}]=[V_{02}+V_{12},V_{02}+V_{12}]=0;$
\item[]
$[X_{i_1,\ldots,i_r},Y_{j_1,\ldots,j_l}]=(-1)^{r(m-r+1)}[Y_{j_1,\ldots,j_l},X_{i_1,\ldots,i_r}]$
\item[]=$\left\{\begin{array}{ll}0,  &\mbox{ }(i_1,\ldots,i_r)\neq
(j_{l+1},\ldots,j_m)\\
\delta'_{m\lambda,-1}\mathrm{sgn}(\tilde{i}_{r+1},\ldots,\tilde{i}_m,\tilde{i}_2,\ldots,\tilde{i}_r)\cdot 1,
&\mbox{ }(i_1,\ldots,i_r)= (j_{l+1},\ldots,j_m).
\end{array}\right.$
\end{itemize}
Then, the semi-direct sum
\begin{itemize}
\item[]${\frak{G}}_1^{SKO}(\underline{t}):={\frak{G}}^{SKO}(\underline{t})\ltimes_{\mathrm{ad}} V^{SKO}$
\end{itemize}
  is a  Lie superalgebra, having  a
$\mathbb{Z}_2$-grading  structure  induced by
$V^{SKO}$, with multiplication
$[M_{m\times \infty}, V^{SKO}]=0$, $\mathrm{ad}f^{SKO}|_{V^{SKO}}=\mathrm{id}_{V^{SKO}}$.

Moreover, the semi-direct sum
\begin{itemize}
\item[]${\frak{G}}_2^{SKO}:={\frak{G}}_1^{SKO}\ltimes_{\mathrm{ad}} \delta'_{m\lambda,-1}\mathbb{F} f_1^{SKO}$
\end{itemize}
  is a  Lie superalgebra, having  a
$\mathbb{Z}_2$-grading  structure  induced by
${\frak{G}}_1^{SKO}$, with multiplication
 $[M_{m\times \infty}, f_1^{SKO}]=[V^{SKO}, f_1^{SKO}]=0$; $[f^{SKO}, f_1^{SKO}]=2f_1^{SKO}$.
Note that ${\frak{G}}_2^{SKO}={\frak{G}}_1^{SKO}$ is a Lie algebra when $\delta'_{m\lambda,-1}=0$.
\end{itemize}
\begin{theorem}\label{ctt1102242003}Let $X=SHO$ or $SKO$. The outer superderivation algebras are as follows:
\begin{itemize}
\item[]$\mathrm{Der_{out}}\big(X(\underline{t})\big)\cong {\frak{G}}^{X}(\underline{t})$;
\item[]$\mathrm{Der_{out}}\big(X(\underline{t})^{(1)}\big)\cong \overline{\frak{G}}_1^{X}(\underline{t})$, when $\underline{t}=\underline{\infty};$
\item[]$\mathrm{Der_{out}}\big(X(\underline{t})^{(i)}\big)\cong {\frak{G}}_i^{X}(\underline{t}),$ $i=1, \; 2$ when $\underline{t}\not=\underline{\infty}.$
\end{itemize}
Moreover, $\mathrm{Der_{out}}\big(X(\underline{t})\big)$, $\mathrm{Der_{out}}\big(X(\underline{\infty})^{(1)}\big)$ and
 $\mathrm{Der_{out}}\big(SKO(\underline{t})^{(i)}\big)$, $i=1, 2$
in the case $\delta'_{m\lambda, -1}=0$ are all Lie algebras. In addition, $\mathrm{Der_{out}}\big(X(\underline{\infty})\big)$ is abelian.
\end{theorem}
 \begin{proof}
For  $\underline{t}\not=\underline{\infty}$, if $m>3$,
 the conclusions follow  directly from  \cite{BLL,Liu-Yuan-1}; if $X=SHO$ and  $m=3$, the conclusions hold from a simple computation. For
$\underline{t}=\underline{\infty}$,
 \begin{itemize}
\item[]
$\overline{SHO}(\underline{t})\cap\overline{S}(\underline{t})=SHO(\underline{t})\oplus \mathbb{F}x_1x_{\tilde{1}}$
\begin{itemize}
\item[]\quad\quad\quad\quad\;\;$=SHO(\underline{t})^{(1)}\oplus \mathbb{F}x^\omega\oplus\mathbb{F}x_1x_{\tilde{1}};$
\end{itemize}
\item[]
$SKO(\underline{t})= \left\{\begin{array}{ll}
SKO(\underline{t})^{(1)}\oplus \mathbb{F}x^\omega&m\lambda-m+2\equiv 0\pmod{p} \mbox{ or } \lambda=1;\\
SKO(\underline{t})^{(1)}&\mbox{otherwise}.
\end{array}\right.$
\end{itemize} From Table 3.1 and Theorem \ref{101123ctc1}, by a direct computation we have the desired results.
 \end{proof}

\begin{remark}\label{ctr1102271257}
In the  finite dimensional simple case   Theorems  \ref{ctt1102241908} and \ref{ctt1102242003} are known \cite{BLL,Liu-Yuan-1,Liu-Zhang-2,Liu-Zhang-Wang,Fu-Zhang-Jiang}.
\end{remark}
\begin{remark}\label{ctr110612rn2}
Note that there is an error in the formulation
of Theorem 2.11 in\cite{Hua-Liu}.
\end{remark}
 \begin{remark}\label{ctr110402}When $\underline{t}<\underline{\infty},$ denote  $\eta=\sum_{i=1}^mt_i$. The following dimension formulas hold:\\
\begin{tabular}{|l|l|l|l|}
   \multicolumn{4}{c}{$\rm{Dimensions \; of\; outer \; superderivations\; algebras\; of\; Cartan\; type}$}

    \\[5pt] \hline
     \multicolumn{1}{|c|}{\small$X$ }&
     \multicolumn{1}{|c|}{\small$\mathrm{\small\dim\Big(\small Der_{out}}\big(X(\underline{t})\big)\Big)$ }&
     \multicolumn{1}{|c|}{\small$\mathrm{\dim\Big(\small Der_{out}}\big(X(\underline{t})^{(1)}\big)\Big) $}&
     \multicolumn{1}{|c|}{\small$\mathrm{\dim\Big(\small Der_{out}}\big(X(\underline{t})^{(2)}\big)\Big)$}\\
     \hline
     {\small$W$}&
      $\small \eta-m$&
      $ $&
      $ $\\
     \hline
     {\small$S$}&
      $ \small\eta-m+1$&
      $\small\eta+1 $&
      $ $\\
      \hline
    {\small $H$}&
      $ \small\eta+1$&
      $\small\eta+2 $&
      $ $\\ \hline
    {\small $ K$}&
      $ \small\eta-m$&
      $\small\eta-m+\delta'_{n-m,3}$&
      $ $\\
       \hline
     \small$HO$&
      $ \small\eta+2$&
      $ $&
      $ $\\
      \hline
     \small$SHO$&
      $\small\eta+2$&
      $\small\eta+2^m+2 $&
      $ \small\eta+2^m+3+\delta_{m=3}$\\
     \hline
     \small$KO$&
      $ \eta-m$&
      $ $&
      $ $\\
      \hline
     \small$SKO$&
      $\small\eta-m+1$&
      $\small\eta-m+1+l(\lambda, m)$&
      $\small\eta-m+1+l(\lambda, m)+\delta'_{m\lambda, -1}$ \\
     \hline
   \end{tabular}
\end{remark}

\end{document}